\documentclass[12pt]{article}
\usepackage{amsmath, amsthm, amssymb, eucal, setspace, mathrsfs, xypic}
\makeatletter
\renewcommand\section{\@startsection{section}{1}{\z@}%
 						{-3.5ex \@plus -1ex \@minus -.2ex}% negative means
						{2ex \@plus.2ex}% 		positive means vertical skip
						{\large\bfseries}}
\renewcommand\subsection{\@ifstar
						{\setcounter{subsection}{\value{equation}}
					\@startsection{subsection}{2}{\z@}
                          {1.75ex \@plus.5ex \@minus.2ex}%
                           {-.4em}		% negative means horizontal
					\textit*}
					{\setcounter{subsection}{\value{equation}}
						\stepcounter{equation}
					\@startsection{subsection}{2}{\z@}
                          {1.75ex \@plus.5ex \@minus.2ex}%
                           {-.4em}		% negative means horizontal
					\textit}}
\def\@seccntformat#1{\@ifundefined{#1@cntformat}%
	{\csname the#1\endcsname\quad} 
	{\csname #1@cntformat\endcsname}} 
\def\section@cntformat{\thesection.~} 
\def\subsection@cntformat{(\thesubsection)\ }
\renewcommand*\l@section{\mdseries\small\@dottedtocline{1}{1.5em}{2em}}
\makeatother

\numberwithin{equation}{section}
\theoremstyle{plain}

\swapnumbers
\newtheorem{theorem}[equation]{Theorem}
\newtheorem{corollary}[equation]{Corollary}
\newtheorem{lemma}[equation]{Lemma}
\newtheorem{proposition}[equation]{Proposition}
\theoremstyle{definition}

\theoremstyle{remark}
\newtheorem{remark}[equation]{Remark}

\newcommand{\cC}{\mathscr{C}}
\newcommand{\dif}{\mathscr{D}}

\newcommand{\cK}{\mathscr{K}}
\newcommand{\cM}{\mathscr{M}}

\newcommand{\cO}{\mathscr{O}}
\newcommand{\cP}{\mathscr{P}}

\newcommand{\cS}{\mathscr{S}}

\newcommand{\frg}{\mathfrak{g}}
\newcommand{\frh}{\mathfrak{h}}

\newcommand{\frl}{\mathfrak{l}}

\newcommand{\frp}{\mathfrak{p}}

\newcommand{\fru}{\mathfrak{u}}

\newcommand{\frB}{\on{Bun}_G}

\newcommand{\frC}{\mathfrak{C}}

\newcommand{\frL}{{\on{Loc}_{^L G}}}

\newcommand{\frO}{{\on{Op}_{^L G}}}

\newcommand{\frS}{\mathfrak{S}}
\newcommand{\frU}{\mathfrak{U}}
\newcommand{\frX}{\mathfrak{X}}
\newcommand{\LD}{{}^L G}
\newcommand{\GL}{\mathrm{GL}}

\newcommand{\bC}{\mathbb{C}}
\newcommand{\bH}{\mathbb{H}}

\newcommand{\bZ}{\mathbb{Z}}
\newcommand{\vep}{\varepsilon}

\newcommand{\on}{\operatorname}
\newcommand{\hg}{\widehat{\mathfrak g}}
\newcommand{\nc}{\newcommand}
\nc{\pparz}{(\!(z)\!)}

\nc{\wh}{\widehat}
								
\begin{document}
\title{\textbf {Geometric Langlands Correspondence Near
    Opers}}
\author{Edward Frenkel and Constantin Teleman}
\date{May 2013}

\maketitle
\begin{center}
\emph{Dedicated to C.S.~Seshadri on his 80th birthday}
\end{center}

\begin{quote}
\abstract{
  \noindent Let $G$ be a complex, connected semi-simple Lie group,
  $\LD$ its Langlands dual group, $\frB$ the moduli stack of
  $G$-bundles on a smooth projective curve $\Sigma$ over $\bC$, $\frL$
  the moduli stack of flat $^L G$-bundles on $\Sigma$. Beilinson and
  Drinfeld have constructed an equivalence between the category of
  coherent sheaves on $\frL$ supported scheme-theoreti\-cally at
  the locus of opers and the category of $\dif$-modules on $\frB$
  admitting a certain global presentation.  We generalize it to an
  equivalence between the derived category of coherent sheaves on
  $\frL$ supported at the formal neighborhood of the locus of opers
  and the localization at $\dif$ of the derived category of
  $\dif$-modules on $\frB$ (and an appropriate equivalence of abelian
  categories).}
%between the thick closure of
%  $\dif\otimes_\cO\cK^{-1/2}$ in $\dif$-modules on $\frB$, and that of
%  the structure sheaf of opers in ${\cC}oh \left(\frL\right)$.
%  Our construction relies on a cohomology computation and an
%  unobstructedness check for deformations.}
\end{quote}

%%%%%%%%%%%%%%%%%%%%%%%%%%%%%%%%%%%%%%%%%%%%%%%%%%%%%%%%%%%%%%%%
\vskip4ex
\begin{minipage}[t]{12cm}
%\tableofcontents
\end{minipage}
\vskip4ex
%%%%%%%%%%%%%%%%%%%%%%%%%%%%%%%%%%%%%%%%%%%%%%%%%%%%%%%%%%%%%%%%

\section{Introduction}

The most complete and satisfying form of the geometric Langlands
correspondence for a complex reductive group $G$ and a smooth
projective complex curve $\Sigma$ has the form of an equivalence of
two categories. One of them is the derived category of quasicoherent
sheaves on $\frL$ (possibly, modified) and the other is the derived
category of $\dif$-modules on $\frB$ (possibly, modified). In the case
that $G$ is the multiplicative group, such an equivalence has been
proved independently by G. Laumon \cite{Laumon:F} and M. Rothstein
\cite{Rothstein}. For general reductive groups, this has been proposed
as a meta-conjecture by A. Beilinson and V. Drinfeld (see, e.g.,
Sects. 4.4, 4.5, 6.2 of \cite{fr} for a survey). Recently, important
progress has been made: a precise conjectural statement of such an
equivalence was proposed by D. Arinkin and D. Gaitsgory \cite{arg},
and its proof was outlined \cite{Gaits} in the simplest non-abelian
case of $G=GL_2$. However, there is still no proof known in the
general case.

In a different direction, using a link between the category of modules
over the affine Kac--Moody algebra $\hg$ at the critical level and the
category of suitably twisted $\dif$-modules on $\frB$, Beilinson and
Drinfeld obtained an equivalence between special pieces in the two
categories \cite{bd} (see Sect.\ 9.5 of \cite{fr} for a survey). On
one side, this is the category of those coherent sheaves on $\frL$
that are supported scheme-theoretically at the locus of \emph{opers}
in $\frL$ (the definition is recalled in \S3 below) and on the other
side, we have the category of $\dif$-modules $\cM$ on $\frB$ admitting
finite global presentations
\[
(\dif\otimes \cK^{-1/2})^{\oplus p} \to(\dif\otimes \cK^{-1/2})^{\oplus q} \to \cM.
\]
Here, $\cK$ is the
canonical line bundle of $\frB$. The appearance of a square root $\cK^{1/2}$ 
(related to the fact that the $\hg$-modules are at the critical level) makes 
it more convenient to switch to the category of
projective $\dif$-modules with the same curvature as $\cK^{1/2}$. We
call it the category of Spin $\dif$-modules. Note that on a smooth 
variety, these are the sheaves of modules over the twisted version
\[
\dif^s:= \cK^{1/2} \otimes_\cO \dif \otimes_\cO \cK^{-1/2}
\] 
of $\dif$, which is the algebra of differential operators on
$\cK^{1/2}$. (On a stack such as $\frB$, this is a bit more subtle.)
But, at any rate, tensoring with the chosen $\cK^{1/2}$ gives an equivalence 
between ordinary and Spin $\dif$-modules.

Note that if $G$ is not simply-connected, there are several choices of
square roots of $\cK$ over $\frB$: on each component of $\frB$, the
choices form a torsor over the $\bZ/2$-valued characters of
$\pi_1(\frB) = H^1(\Sigma; \pi_1G)$. However, a preferred choice of
$\cK^{1/2}$, the \emph{Pfaffian} of the cohomology of the universal
adjoint bundle over $\frB$, arises from a choice of a
theta-characteristic (Spin structure) on $\Sigma$: see \cite{ls} for
the construction. The same choice of theta-characteristic appears on
the Langlands dual side, in the definition of opers. The relevant
categories for different theta-characteristics are equivalent.

The construction of \cite{bd} relies on the double quotient
realization
\[
\frB \simeq G[\Sigma^o] \backslash G\pparz/G[[z]],
\]
where $\Sigma^o=\Sigma\setminus\{\sigma\}$, and on the description of the
center of a completion of $U(\hg)$ at the critical level due to
B.~Feigin and the first author \cite{ff}. We will recall some needed
definitions below, though our discussion in Sections 2 and 3 of the
background to the problem will be brief; a more detailed discussion
may be found in Part III of \cite{fr}.

It is natural to ask whether the Beilinson--Drinfeld construction
\cite{bd} could be extended to an equivalence of larger subcategories
on both sides of the Langlands correspondence. The main difficulty in
doing this is that, unlike the flag varieties of Lie groups (and
conjecturally of loop groups), the double quotient stack $\frB$ is
\emph{not} $\dif$-affine; therefore, the functor of global sections is
not faithful. Faithfulness fails even at the level of derived
categories, which are sometimes more forgiving in this respect. The
problem become even more pronounced on the Langlands dual side, where
the $\dif^s$-module $\dif^s$ is expected to correspond to the
(possibly degree-shifted, depending on one's conventions) structure
sheaf $\cO_\frO$ of the subvariety of opers in $\frL$. The functor
$\mathbf{R}\Gamma\left(\frB; \underline{\:\:\:}\right)$, on
$\dif^s$-modules corresponds to $\mathbf{R}\mathrm{Hom}(\cO_\frO;
\underline{\:\:\:})$ on $\cC oh (\frL)$, and the latter loses all
information away from $\frO$.

Nevertheless, in this note we extend the Beilinson--Drinfeld
construction to the formal neighborhood of the locus $\frO$ of $^L
G$-opers inside $\frL$. Our first step is the computation of the
derived endomorphism algebra of the structure sheaf $\dif^s$, and its
identification with the $\mathrm{Ext}$-algebra of the structure sheaf
$\cO_\frO$ of opers in $\frL$. Both algebras are strictly commutative
($A_\infty$-equivalent to strictly commutative algebras), and this
suffices to establish an (abstract) equivalence of formal deformation
theories of $\dif^s$, within the category of $\dif^s$-modules on
$\frB$, with that of $\frO$ within $\frL$. On general grounds, this
implies an (abstract) extension of the Beilinson--Drinfeld
construction to the formal neighborhoods of $\dif^s$ and $\cO_\frO$ in
the corresponding categories and derived categories. (The
$\mathrm{Ext}$-algebras are the \emph{Koszul duals}, relative to
$\dif^s$, respectively $\frO$, of the formal neighborhoods of the two
objects.)

An interesting question is to find a \emph{canonical} isomorphism
between the respective categories. Indeed, the original construction
in \cite{bd} is essentially canonical, and it intertwines the Hecke
functors on the $\frB$ side and certain functors on the $\frL$ side
(see also \cite{arg}). Recall that Hecke functors $H_{x,V}$ are
parametrized by points $x\in\Sigma$ and representations $V$ of $\LD$.
A Hecke eigensheaf is a $\dif^s$-module $\cM$ which, under the action
of $H_{x,V}$, gets tensored with the fiber at $x$ of the flat
$V$-bundle over $\Sigma$ associated to a particular $\LD$-local
system. That local system is then the ``Hecke eigenvalue'' of $\cM$, a
point in $\frL$. This pins down the Langlands correspondence at the
level of generic points of $\frL$.  Beilinson and Drinfeld identify
the algebra of global sections $\Gamma(\frB;\dif^s)$ with the
functions on $\frO$, and the resulting ``decomposition'' of the sheaf
$\dif^s$ over $\frO\subset\frL$ turns out to be nothing but the
spectral decomposition under the Hecke action.

In the present paper, our generalization of their construction is only
canonical to first order: first-order deformations of $\dif^s$ in the
category $\dif^s_{\frB}\mbox{-mod}$ are canonically identified with
those of $\frO$ in $\frL$. Since the eigensheaf quotients of $\dif^s$
correspond to points of $\frO$, we believe that they should deform to
Hecke eigensheaves to all orders,\footnote {This would be automatic if
  we knew that these eigensheaves are irreducible, but this is not
  known at present.} with their Hecke eigenvalues being the formal
paths in $\frL$ starting at points in $\frO$. Confirming this, and
showing that we thus extend the Beilinson--Drinfeld \emph{spectral}
decomposition formally to all orders, requires a finer analysis which
will be made in a follow-up paper.

Our main result in this paper, the cohomology calculation of
\S\ref{coh}, goes back some 15 years. Although we lectured on it on
numerous occasions, it took us a long time to write it up. Our
reluctance was due to the tantalizing proximity of a much stronger
result which would follow, if only some general facts about Hecke
eigensheaves were known.  Indeed, our cohomology calculation and
unobstructedness results provide the Jacobian test for \'etaleness of
the ``Hecke eigenvalue'' morphism, from the moduli stack of simple
Hecke eigensheaves to $\frL$. The categorical Langlands correspondence
predicts an isomorphism\footnote{ This prediction must be adjusted at
  the singular points of $\frL$, see \cite{arg}.}  between the moduli
stack of the former objects and $\frL$; furthermore, at the level of
points, it has been largely proved for $\mathrm{GL}(n)$ \cite{fgv,
  gaits-gln}. If we had known {\em a priori} that the aforementioned
moduli stack were locally of finite type, and that the Hecke
eigenvalue morphism were algebraic, its \'etale property near $\dif^s$
would follow.

Unfortunately, these hoped-for properties of the moduli of Hecke
eigen-$\dif^s$-modules are still unsettled. What's even worse is that
it is not known whether the Hecke eigensheaves constructed in
\cite{bd} are in fact irreducible.\footnote{A general argument can be
  used to show that a \emph{generic} module in their family is
  irreducible.}  This is why the extent of our results is not quite
what we had originally hoped for. We do have a non-trivial statement
nonetheless, the one about the formal neighborhood of opers, which we
present here.

\smallskip

It is a special pleasure to contribute this paper to a collection
honoring C.~S.~Seshadri, whose work on vector and principal bundles on
the Riemann surfaces has been a cornerstone for so much rich and
beautiful mathematics developed over the past few decades.

\smallskip

The work of E.F. was supported by DARPA under the grant
HR0011-09-1-0015 and by NSF under the grant DMS-1201335. C.T. was
supported by NSF under the grant DMS-1007255. Both authors were
supported by the NSF grant DMS-1160328.

\smallskip

\begin{remark}
  To avoid any confusion, we note that there are actually \emph{two}
  different versions of the categorical Langlands
  correspondence. Indeed, recent work by Kapustin and Witten \cite{kw}
  in 4D gauge theory also predicts an equivalence of categories, which
  is purely topological in nature and is not sensitive to the
  algebraic structure of $\Sigma$. It is different from the one we
  discuss here (and the one in \cite{arg}). The reason why there are
  are two equivalences is that there are two algebraic structures on
  $\frL$, the \emph{Betti} structure, which identifies $\frL$ with the
  representation variety of $\pi_1(\Sigma)$ in $\LD$ and the \emph{de
    Rham} structure of flat bundles (they are linked pointwise via the
  monodromy of a connection). The former is not sensitive to the
  algebraic structure of $\Sigma$, but the latter is -- and in fact,
  $\Sigma$ can be recovered from it. Correspondingly, there are two
  algebraic structures on the moduli of (regular holonomic)
  $\dif$-modules with singular support in Laumon's global nilpotent
  cone, which are related by the Riemann-Hilbert correspondence: the
  Betti structure, seeing the monodromy representation and extension
  data between the various strata, and the de Rham structure arising
  from the underlying $\dif$-module. The second is definitely
  sensitive to the algebraic structure on $\Sigma$, whereas the former
  is not believed to be so. These algebraic structures are meant to
  match pairwise under the Langlands correspondence.  Indeed, any
  natural construction of the correspondence would make such a match,
  but we don't have a natural construction of the correspondence at
  present.
\end{remark}

\section{Warm-up: the case $G=\LD=\GL(1)$}

We describe the $\GL(1)$ case of the geometric Langlands equivalence,
due to Laumon \cite{Laumon:F} and Rothstein \cite{Rothstein}, and in
the process identify the portion which corresponds to the
Beilinson--Drinfeld construction.  Here, we opt for a simplification
which ignores the global automorphism group $\GL(1)$ of line bundles
over $\Sigma$, as well as the group of components of
$\mathrm{Pic}(\Sigma)$. (These two groups are interchanged under
the Langlands duality.)

\subsection{Fourier-Mukai transform.}
Let $A$ be an Abelian variety and $A^\vee$ its dual, $\cP\to A\times
A^\vee$ the Poincar\'e bundle and $\pi, \pi^\vee$ the projections of
the product to the two factors.  The \emph{Fourier-Mukai transform}
\[
\Phi: \mathbf{D}\cC oh(A) \to \mathbf{D}\cC oh(A^\vee), \quad \Phi(\cS): = 
	\mathrm{R}\pi^\vee_*(\cP\otimes\pi^* \cS) 
\]
is an equivalence of the derived categories (with arbitrary
decorations: $+,-,b$).

Following Laumon and Rothstein, let us enhance this as follows. The
cotangent bundle $T^\vee A$ is trivial, projecting by a map $p$ to the
vector space $\mathbf{H}: = T^\vee_0A$.  There is a similar projection
$A^\vee\times \mathbf{H} \to \mathbf{H}$. We can construct a
Fourier-Mukai equivalence $\mathbf{D}\cC oh(T^\vee A) \cong \mathbf{D}
\cC oh(A^\vee\times \mathbf{H})$ \emph{relative to $p$} by the same
formula, from the correspondence diagram
\[
\xymatrix{
 & T^\vee A\times A^\vee  \ar[dl]_{T^\vee\pi} \ar[dr]^{p\times\mathrm{Id}} & \\
T^\vee A & & \mathbf{H} \times A^\vee
}
\]

\subsection{Deformation.} \label{abdef} There is a distinguished
deformation of $\mathbf{D}\cC oh (T^\vee A)$, the non-commutative
deformation of $T^\vee A$ defined by the natural symplectic form
(`quantization'). One natural implementation as a genuine deformation,
as opposed to a formal one, is the category of coherent $\dif$-modules
over $A$. There would be a $1$-parameter deformation, but the scaling
of the cotangent bundle identifies all non-zero deformed categories in
the family.

Recall now that $T_0A^\vee$ is naturally identified with the
\emph{complex conjugate} of $\mathbf{H} = T^\vee_0A$. There results a
$1$-parameter deformation $A^\flat$ of the space $A^\vee\times
\mathbf{H}$ --- this time, in the commutative world --- to an affine
bundle over $A^\vee$, classified by $\mathrm{Id}\in H^1(A^\vee;
\cO\otimes \mathbf{H})$. The key observation of Laumon and Rothstein
is that the Fourier-Mukai equivalence relative to $p$ deforms to an
equivalence between the deformed categories
$\mathbf{D}_{coh}\dif_A\mbox{-mod}$ and $\mathbf{D}\cC
oh(A^\flat)$. (Here $\mathbf{D}_{coh}\dif_A\mbox{-mod}$ is the derived
category of coherent $\dif_A$-modules.)

It is important to note that $A^\flat$ \emph{no longer} projects onto
$\mathbf{H}$; instead, the fiber of $A^\flat\to A^\vee$ over $0\in
A^\vee$ is identified with $\mathbf{H}$, as a ``section'' of the no
longer existent projection. Intrinsically, $A^\flat$ is the moduli of
\emph{flat} holomorphic line bundles on $A$, and the Poincar\'e bundle
$\cP$ deforms to the universal bundle $\cP'\to A\times A^\flat$, which
has a flat connection along the $A$-factor. The deformed
Fourier-Mukai transform uses the push-forward of the \emph{de Rham
  complex} of $\cP'\otimes \cM$:
\[
\Phi^\flat: \mathbf{D}\dif_A\mbox{-mod} \to \mathbf{D}\cC oh(A^\flat),
\quad \Phi^\flat(\cM) :=
	\mathrm{R}\pi^\vee_*
        \left(\Omega^\bullet(\cP'\otimes\pi^*\cM),d\right).
\] 
A homological algebra exercise shows that this construction becomes
the old $\Phi$ relative to $p$, when the deformation degenerates back
to the ``classical'' product case.

\subsection{The abelian Beilinson--Drinfeld construction.}   
Let now $A$ be the Jacobian $\mathrm{Pic}^0$ of our curve $\Sigma$. In
this case, $A^\flat$ can also be identified with the moduli space of
$\GL(1)$-local systems on $\Sigma$.  Note that, as $A$ is principally
polarized, $A$ is isomorphic to $A^\vee$, but this does not \emph{yet}
play a role.\footnote{Replacing $\GL(1)$ by a torus would make
  $A^\vee$ into the moduli space of principal bundles for the dual
  torus.}  Let $\frO\subset A^\flat$ be the fiber over $0$, which is
the space $H^0(\Sigma;\Omega^1)$ of differentials; let us call it the
\emph{space of $\GL(1)$-opers}.

One easily determines that $\Phi^\flat(\dif) = \cO_\frO[-\dim A]$, and that 
\begin{equation}\label{abelianend}
  \mathrm{End}_{\dif_A}(\dif) = \Gamma(A;\dif) =\mathrm{Sym}\, T_0A =
  \bC[\frO] =  \mathrm{End}_{\cC oh(\frL)}( \cO_\frO).
\end{equation}
Exploiting this and the flatness of $\dif$ over its global sections,
we obtain the following

\begin{theorem}\label{abelianbd}
\begin{enumerate}
\item Any $\dif$-module $\cM$ which admits a finite free global
  presentation $\dif^{\oplus p} \to \dif^{\oplus q} \to \cM$ has a
  finite free global $\dif$-resolution.
\item The Fourier transform $\Phi^\flat(\cM)$ of such a module is a
  coherent sheaf supported on $\frO$, shifted into degree $\dim A$.
\item $\Phi^\flat$ restricts to an equivalence between the abelian
  category of $\dif$-modules with free global presentations to $\cC
  oh(\frO)$.  Thereunder, the functor $\Gamma(A;\underline{\:\:\:})$
  corresponds to $\Gamma(\frO; \underline{\:\:\:})$.
\item The structure sheaf $\bC_\omega$ of a point $\omega\in\frO$ corresponds 
to the $\dif$-module $\dif\otimes_{\Gamma(\dif)}\bC_\omega$.
\end{enumerate}
 \end{theorem}

 \noindent Note, in passing, that the $\dif$-modules in the theorem
 all have \emph{trivial} underlying vector bundles.
\begin{remark}
  The $\dif$-modules $\dif\otimes_{\Gamma(\dif)}\bC_\omega$ have the
  following \emph{Hecke eigensheaf property}. A degree zero divisor
  $D\subset\Sigma$ defines a line $\omega^{D}$, the tensor product of
  the fibers over $D\subset\Sigma$ of the flat line bundle classified
  by $\omega$. Let $T_D:A\to A$ denote the translation by
  $\cO(D)$. (these translations are the abelian Hecke
  correspondences.) There exist, then, functorial isomorphisms
\[
T_D^* \left(\dif\otimes_{\Gamma(\dif)}\bC_\omega\right) 
		\equiv
                \left(\dif\otimes_{\Gamma(\dif)}\bC_\omega\right)
                \otimes \omega^D,
\]
coherently additive in $D$.
\end{remark}

\subsection{Derived version.} 
The equalities \eqref{abelianend} enhance to 
\begin{equation}\label{derivedabelianend}
  \mathbf{R}\mathrm{End}_{\dif_A}(\dif) \equiv \mathbf{R}\Gamma(A;\dif) 
  \equiv \mathrm{Sym}\,T_0A\otimes \bigwedge\nolimits^\bullet
  \overline{T_0^\vee A} \cong
  \Omega^\bullet[\frO] \cong \mathbf{R}\mathrm{End}_{\cC
    oh(A^\flat)}(\cO_\frO).
\end{equation}
We have written $\equiv$ for canonical isomorphisms; whereas, in the
isomorphisms with the algebra
$\mathrm{Sym}\,T_0A\otimes\bigwedge^\bullet T_0 A$ of differential
forms on $\frO$, we have invoked the polarization of $A$. We did this
twice, left and right: all other terms are canonically isomorphic,
because $\overline{T_0^\vee A}$ is the normal bundle to $\frO$ in
$A^\flat$.

There is more to the isomorphism \eqref{derivedabelianend} than first
meets the eye, because the two $\mathbf{R}\mathrm{End}$ complexes are
more than algebra objects in the derived categories (of vector
spaces): namely, they have natural $A_\infty$-structures, from the
(higher) Yoneda products. It is with these higher structure included
that they are isomorphic, as $A_\infty$-algebras, to the
skew-commutative algebra $\Omega^\bullet [\frO]$.\footnote{The linear
  structures on $A, A^\flat$ allow one to verify that; but it also
  follows from the fact that the Fourier transform $\Phi^\flat$
  enhances to a functor between the differential graded categories
  underlying the derived ones.}

More generally, for any $\cM\in\mathbf{D}(\dif_A\mbox{-mod})$, we
find, by invoking the results of Laumon and Rothstein, that
\begin{equation}\label{derivedabelianhom}
\mathbf{R}\mathrm{Hom}_{\dif_A}(\dif;\cM) = \mathbf{R}\Gamma(A;\cM) = 
	\mathbf{R}\mathrm{Hom}_\frL\left(\cO_\frO;
          \Phi^\flat(\cM)\right)[\dim A].
\end{equation}
Local Serre duality identifies the latter with the global sections of
the derived restriction of $\Phi^\flat(\cM)$ to $\frO$. In this sense,
the functor $\mathbf{R}\Gamma$ on $\mathbf{D}_{coh} \dif_A\mbox{-mod}$
corresponds to the restriction to $\frO$ in $\mathbf{D}\cC
oh(A^\flat)$ (followed by global sections, which however do no
``damage'' as $\frO$ is affine).

The derived version of the abelian Beilinson--Drinfeld correspondence
restricts the Fourier transform $\Phi^\flat$ to the part which can be
gleaned from \eqref{derivedabelianend}, \eqref{derivedabelianhom}.
There are variations of the statement: we can use the abelian
categories or the derived categories, or we can use subcategories ---
thick closures of $\dif$ or $\cC oh_\frO$, respectively --- or
localize the categories at the respective objects.

\begin{theorem}
\begin{enumerate}
\item $\Phi^\flat$ restricts to an equivalence of triangulated
  categories, from the thick closure of $\dif$ in
  $\mathbf{D}(\dif_A\mbox{-mod})$ to that of $\cO_\frO$ in
  $\mathbf{D}\left(\cC oh(\frL)\right)$.
\item Restricting to the corresponding abelian categories, we get an
  equivalence between the categories of nilpotent deformations of
  objects in Theorem~\ref{abelianbd},(iii).
\item The same holds for the localizations at $\dif$ and $\cO_\frO$ of
  the abelian or derived categories of $\dif_A$ and $\cC
  oh(\frL)$-modules, respectively.
\end{enumerate}
\end{theorem}

\begin{remark}
  Recall that the localization ${}_x\frC$ of a ($\bC$-linear) category
  $\frC$ at an object $x$ is obtained by inverting morphisms
  $\varphi:y\to z$ which induce isomorphisms $\varphi_*:
  \mathrm{Hom}(x,y) \to \mathrm{Hom}(x,z)$. The functor ${}_x\frC \to
  Mod-\mathrm{End}(x)$, $y\in \mathrm{Ob}(\frC)\mapsto
  \mathrm{Hom}(x,y)$, is faithful; often, as in our examples, it can
  be made full and essentially surjective, with the right finiteness
  conditions on $\mathrm{End}(x)$-modules. Thus, $\cC oh(\frL)$
  localized at $\cO_\frO$ is equivalent to the category of coherent
  sheaves on the formal neighborhood of $\frO$ in $\frL$.  On the
  other hand, the thick closure of $\cO_\frO$ in $\cC oh(\frL)$
  comprises all successive extensions of $\cO_\frO$-modules in $\cC
  oh(\frL)$, the coherent sheaves supported in some finite-order
  neighborhood of $\frO$.
\end{remark}

\section{The constructions of Hitchin and of Beilinson--Drinfeld}
We now recall the (semi-)classical construction of N. Hitchin
\cite{hit} and its quantization due to Beilinson--Drinfeld. One
substantial difference with the abelian case is the stack nature of
the moduli $\frB$ of $G$-bundles, which cannot be avoided because
different $G$-bundles have different automorphism groups. (We cannot
restrict ourselves to the moduli spaces of semi-stable bundles because
they are not preserved by the Hecke correspondences, which are
essential for the formulation of the geometric Langlands
correspondence.)

%difficult to avoid, because
%the singularity of the associated moduli space of semi-stable bundles
%(and Higgs bundles). In the section following this one, where higher
%cohomology appears, use of the stack is essential.

\subsection{The cotangent stack.}\label{cotstack}
The ``classical'' version of $\dif$ (or of any of its twisted
versions) on a manifold $X$ is the sheaf of $\cO_X$-algebras
$\mathrm{Sym}\, T$. For a smooth stack $\frX$, which is locally a
quotient of a smooth manifold by an algebraic group, $T$ is a complex
defined up to a quasi-isomorphism, and there is the corresponding
differential graded version of $\mathrm{Sym}\,T$. Locally, if $\frX =
X/L$, a model for the tangent complex $T\frX$ is $[\frl
\xrightarrow{\:\alpha\:} TX]$ (in degrees $-1$ and $0$) and the
infinitesimal action map $\alpha$. Accordingly,
\[
\mathrm{Sym}^r\, T\frX \cong
\left(\bigoplus_{s+t=r}\textstyle{\bigwedge}^{s}\frl\otimes
  \mathrm{Sym}^t\,TX, \partial_\alpha\right).
\]
The differential induced from $\alpha$ makes the total $\mathrm{Sym}$
into a differential graded algebra.  The \emph{cotangent stack}
$T^\vee\frX:=\mathrm{Spec}_\frX\,\mathrm{Sym}\,T\frX$ is a derived
stack (in vector spaces) over $\frX$. Different presentations of
$\frX$ give equivalent models for $T^\vee\frX$.

%in a sense that we will let the reader spell out.

This applies for instance to the stack $\frX=\frB$, with some bad and
some good news.
\begin{itemize}
\item Bad news: $\frB$ may only be presented as a quotient locally (on
  substacks of finite type).
\item Good news: when $G$ is semi-simple and $\Sigma$ has genus $2$ or
  more, $\mathrm{Sym}\,T\frB$ is quasi-isomorphic to its degree zero
  (top) cohomology; so it is not truly a derived stack.
\end{itemize}
The good news overrides the bad, as it gives a strict model for
$T^\vee\frB$ over $\frB$: we need not delve into coherent systems 
of quasi-isomorphisms in patching the local descriptions together. 

\medskip

\emph{ Henceforth, unless specified otherwise, we will assume that $G$
  is semi-simple and the genus of $\Sigma$ is $2$ or more.}

\subsection{The Hitchin map.} 
Hitchin \cite{hit} constructed a morphism $\chi$ from $T^\vee\frB$ to
a ${\mathbb Z}$-graded vector space $\mathbf{H}$, generalizing the
projection $p:T^\vee A \to \mathbf{H}$ we used in \S2 for
$G=\GL(1)$. (Hitchin originally introduced $\chi$ for the moduli
\emph{space} $[T^\vee\frB]$ of stable objects associated to
$T^\vee\frB$, the stable \emph{Higgs bundles}; but extending it to the
stack is straightforward.)  The space $\mathbf{H}$, nowadays called
the \emph{Hitchin base}, is isomorphic to
\[
\mathbf{H} = \bigoplus_{m\in\exp(\frg)} \mathbf{H}^{(m+1)} \simeq
\bigoplus_{m\in\exp(\frg)} H^0\left(\Sigma; K_\Sigma^{\otimes
    (m+1)}\right),
\]
with the direct sum ranging over the exponents of $\frg$. This
isomorphism is not canonical and involves the choice of a set of
graded generators of the space of invariant functions on the Lie
algebra ${\mathfrak g}$.\footnote{It is easy to describe $\mathbf{H}$
  canonically in terms of the graded vector space
  $\on{Spec}(\bC[{\mathfrak g}]^G)$, but the above description will
  suffice for us.}  The fiberwise ${\mathbb G}_m$-action on
$T^\vee\frB$ is compatible with the ${\mathbb G}_m$-action on
$\mathbf{H}$ defined by the ${\mathbb Z}$-grading.  Hitchin proved
that $\chi$ was flat, that its generic fiber within the neutral
component of $[T^\vee\frB]$ was an abelian variety, and that it was a
torsor over it within every other component; finally, that these
fibers were Lagrangian for the natural symplectic structure on
$[T^\vee\frB]$. These results were extended in \cite{bd} from the
moduli space of stable bundles to the stack $T^\vee\frB$. One
important distinction here is the appearance of a gerbe with band
$Z(G)$ over the regular fibers of $\chi$ in the stack $T^\vee\frB$,
coming from global automorphisms of $G$-bundles.

The map $\chi$ has a section, which depends on the choice of a Spin
structure on $\Sigma$.  This section is the classical limit of the
Lagrangian variety of opers, recalled below, and is relevant to the
Langlands dual Hitchin fibration ${}^L\chi:T^\vee\frB(\Sigma;\LD) \to
{}^L\mathbf{H}$. (Remarkably enough, the bases $\mathbf{H}$ an
${}^L\mathbf{H}$ can be identified, once we choose an invariant
quadratic form $k$ on $\frg$.)

For $G=\GL(1)$, this slice is the fiber $0\times \mathbf{H}$ from the
previous section. For $G=\mathrm{SL}(2)$, the slice sends
$\mathbf{h}\in H^0(\Sigma; K_\Sigma^{\otimes 2})$ to the pair
consisting of the bundle $E:= K_\Sigma^{1/2}\oplus K_\Sigma^{-1/2}$
and the cotangent vector (``Higgs field")
\[
\left[\begin{array}{cc} 0 & \mathbf{h} \\ 1 & 0\end{array}\right]:
E\to E\otimes K_\Sigma.
\]
For general $\frg$, the slice can be constructed using the principal
embedding $\mathfrak{sl}(2) \to \frg$ and the Kostant slice for the
adjoint action.

Choose now a non-degenerate invariant quadratic form $k$ on $\frg$,
and a corresponding ample line bundle $\Theta_k$ on $\frB$. (If
$\frg$ is a simple Lie algebra, then the form $k$ is unique up to
scalar; there can be some torsion ambiguity in defining $\Theta_k$,
but this will not be relevant.)  Hitchin proved the following for the
moduli space of stable Higgs bundles. With our assumptions on $G$ and
$\Sigma$, the codimension of the unstable part of the stack is too
high to have an effect, so we state the result directly for the stack.

For each $\gamma \in \pi_1 G$, let $\frB^\gamma$ be the corresponding
component of $\frB$.

\begin{theorem}[\cite{hit}]\label{hitchinthm}
Assume that $G$ is simply-connected. Then,
\begin{enumerate} 
\item All global functions on $T^\vee\frB^\gamma$ are lifted from the
  base: $H^0(T^\vee\frB^\gamma;\cO) \cong \bC[\mathbf{H}]$.
\item $H^1(T^\vee\frB^\gamma;\cO) \cong \bC[\mathbf{H}]\otimes
  \mathbf{H}^\vee \cong \Omega^1[\mathbf{H}]$.
\item Specifically, the generators of $H^1$ over $H^0$ arise from
  $H^0$ by contracting the vector fields defined by the linear
  Hamiltonians with $c_1(\Theta_k)\in H^1(T^\vee\frB;\Omega^1)$.
\end{enumerate}
%For general $G$, the results apply to every component of $\frB$
%separately.
\end{theorem}

\begin{remark}
  The tangent bundles to the regular fibers of $\chi$ may be
  identified with $\mathbf{H}^\vee$ using the symplectic structure on
  $T^\vee\frB$. The fibers of $R^1\chi_*(\cO)$ are therefore
  isomorphic to $\overline{\mathbf{H}}$; we identify them with
  $\mathbf{H}^\vee$ by using $c_1(\Theta_k)$.
\end{remark}

The result on $H^0$ immediately implies an equivalence between the
abelian category of $\cO$-modules on $T^\vee\frB$ with a finite global
presentation by $\cO$ and the abelian category of coherent sheaves
with support on the Hitchin slice in $T^\vee\on{Bun}_{\LD}$. The
components of $T^\vee\frB$ labeled by $\pi_1G$ get interchanged with
the characters of the $Z(\LD)$-gerbe over the slice. In
Theorem~\ref{classical} below, we will generalize Hitchin's result to
all cohomologies; this will affirm the derived versions of this
equivalence.

\subsection{Deformations.} Hitchin's construction extends to
deformations of the stacks $T^\vee\frB$ and $T^\vee\on{Bun}_{\LD}$ in
the manner analogous to \S\ref{abdef}.  It is easier to describe the
analogue of $A^\flat$, in which the connected components of the stack
$T^\vee\on{Bun}_{\LD}$, fibered in vector spaces over the components
of $\on{Bun}_{\LD}$, deform into fibrations in \emph{affine} spaces,
according to the class $[k^{-1}]\in
H^1\left(\on{Bun}_{\LD};\Omega^1\right)$. (Recall that this group is
also $H^2\left(\frB(\Sigma;\LD);\bC\right)$, and the quadratic form
$k^{-1}$ on ${}^L\frg$ dual to $k$ on $\frg$ defines a class in the
latter.)  The total stack of this relative affine deformation is
isomorphic to the stack $\frL$ of $\LD$-local systems.

The relevant non-commutative deformation of $T^\vee\frB$ converts
$\mathrm{Sym}\,T$ into the sheaf $\dif^s$ of Spin differential
operators\footnote{One can deform to any twisted form of
$\dif$, but only the Spin version gives an interesting answer to the
questions we are interested in.} (differential operators on $\cK^{1/2}$). 
This has an increasing filtration whose associated graded is
$\mathrm{Sym}\,T\frB$.

\subsection{Technical aside on $\dif$ of a stack.} \label{diffstack} A
(derived) category of $\dif$-modules on an Artin stack $\frX$ can be
defined in a number of ways.  We can use crystals, or (complexes of)
$\cO$-modules over the de Rham site $dR(\frX)$ \cite{st}.  The sheaf
$\dif_\frX$ of differential operators is more peculiar. It is a
distinguished object in $\mathbf{D}(\dif_\frX\mbox{-mod})$; in a local
presentation $\frX=X/L$, a model for its underlying $\cO$-module is
the Chevalley complex of $\dif_X$ and the Lie algebra $\frl$. This has a
natural increasing filtration, whose associated graded is
$\mathrm{Sym}\, [\frl\xrightarrow{\alpha} TX]$ of \S\ref{cotstack}.

This $\dif_\frX$ is \emph{not} a sheaf of algebras over the smooth
site of $\frX$. Even for a manifold $X$, the extension of $\dif_X$ to
the smooth site by pull-backs is \emph{not} a sheaf of
\emph{algebras}. However, over any open \emph{substack}
$\frU\subset\frX$, $H^\bullet (\frU; \dif)$ is an algebra, isomorphic
to $\mathrm{Ext}^\bullet_{\dif(\frU)}(\dif, \dif)$.  One
characterization of $\dif$, which makes the isomorphism automatic, is
$p_!\cO$, where $p_!$ is the (derived) left adjoint of the pull-back
along the projection $p:\frX \to dR(\frX)$ \cite{bd}.  The
$\mathrm{Ext}$-description enhances to give an $A_\infty$-algebra
structure with the Yoneda products.

\emph{Twisted analytic versions} of $\dif_\frX$ are classified by
$H^1(\frX; \cO_{an}/\bC)$. On a manifold, $\cO_{an}/\bC\cong
(\cO_{an})^\times/\bC^\times$ is the sheaf of algebra automorphisms of
$\dif$ fixing $\cO$, and $1$-cocycles allow the patching together of
local copies of $\dif$ by means of isomorphisms fixing $\cO$. The same
group is the analytic hyper-cohomology $\bH^2\left(\frX;
  (\Omega_\frX^{\ge 1},d)\right)$ of the truncated de Rham
complex. When Hodge decomposition holds --- or else, in the algebraic
case, with log poles at infinity --- the latter becomes the Hodge
filtered part $F^{\ge 1}H^2(\frX;\bC)$.

The algebraic version of the same hyper-cohomology defines
\emph{twisted algebraic} versions of $\dif$. The latter are \emph{not}
locally isomorphic to $\dif$ in the Zariski topology, but to twisted
versions $\dif^\varphi$ satisfying $[X,Y] = \varphi(X,Y)$ for vector
fields $X,Y$ and (locally defined) closed $2$-forms $\varphi$. A class
in $\bH^2\left(\frX;(\Omega_\frX^{\ge 1},d)\right)$ is represented, in
a \v{C}ech covering $U_i$, by $1$-forms $\eta_{ij}$ on $U_{ij}$ and
closed $2$-forms $\varphi_i$ on $U_i$, satisfying the \v{C}ech cocycle
condition for the $\eta$ and the relations $d\eta_{ij} =
\varphi_i-\varphi_j$. The transformation $X\mapsto X+\eta_{ij}(X)$
extends to algebra isomorphisms patching together the
$\dif^{\varphi_i}$. Our situation is simplified by the isomorphism
$H^2(\frB;\bC) = H^1(\frB;\Omega^1)$, for semi-simple $G$; the twisted
versions of $\dif$ over $\frB$ are classified by invariant quadratic
forms on $\frg$, and are interpreted as differential operators on
(complex) powers of the theta-line bundle.

\subsection{Opers.} The final ingredient in the quantization of
Hitchin's construction is the space of \emph{opers} in $\frL$. There
are several isomorphic components, and just like for the Hitchin
section, pinning down one of them requires a choice of
$K_\Sigma^{1/2}$.  For $\LD=\mathrm{PSL}(2)$, the opers are the local
systems whose underlying $\mathrm{SL}(2)$-bundle is the unique (up to
isomorphism) non-trivial extension
\[
K_\Sigma^{1/2} \to E \to K_\Sigma^{-1/2}.
\]
For general $\LD$, they are the local systems whose underlying
principal bundle $P$ is defined from $E$ by the inclusion of the
principal $\mathrm{SL}(2)$ subgroup in $\LD$. Recall now that $\frL$
has an algebraic symplectic structure, defined by the pairing on forms
in any invariant quadratic form on ${}^L\frg$.

\smallskip

\emph{In a break with the traditional terminology and notation, in
  this paper we will denote by $\frO\subset \frL$ one of the
  components of the space of opers (specifying this component may
  require a choice of $K_\Sigma^{1/2}$).}

\smallskip

Beilinson and Drinfeld have proved \cite{bd} the following:

\begin{proposition}\label{operfilt}
  $\frO\subset\frL$ is a smooth Lagrangian subvariety contained in
  the smooth, Deligne--Mumford part of $\frL$. It carries a trivial
  gerbe with band $Z(\LD)$, and the associated space is an affine
  space over $\mathbf{H}$.
\end{proposition}
\noindent The last part follows from the computation 
\[
\Gamma(\Sigma,ad_P\otimes K_\Sigma) \cong \mathbf{H}, 
\]
which uses the Kostant slice for the adjoint action. When combined
with the grading on $\mathbf{H}$, the proposition supplies an
increasing filtration on the algebra of functions $\bC[\frO]$.

Recall that the components $\frB^\gamma$ of $\frB$ are labeled by
$\gamma \in \pi_1G$.

\begin{theorem}[\cite{bd}]\label{bdthm}
  There is a canonical isomorphism of algebras $\Gamma(\frB^\gamma,
  \dif^s) \equiv \bC[\frO]$ of global Spin differential operators on
  $\frB^\gamma$ with the polynomial functions on $\frO$.
\end{theorem}

\begin{remark}
  Beilinson and Drinfeld prove more: for a point $\omega\in \frO$,
  with residue field $\bC_\omega$, the $\dif^s$-module
  $\dif^s_\omega:= \dif^s\otimes_{\Gamma(\dif^s)}\bC_\omega$ is
  regular holonomic, with singular support in Laumon's nilpotent cone,
  and is a Hecke eigensheaf whose eigenvalue is the local system
  $\omega$. Thus, $\dif^s_\omega$ is assigned to $\omega$ under the
  geometric Langlands correspondence.
\end{remark}

As in our discussion for Theorem~\ref{abelianbd}, this gives 
\begin{corollary}\label{deformationcor}
\begin{enumerate}
\item Any $\dif^s$-module admitting a finite free global presentation
  by copies of $\dif^s$ has a finite free global $\dif^s$-resolution.
\item The corresponding presentation in $\cC oh(\frO)$ defines a
  coherent sheaf, and this assignment gives an equivalence from the
  abelian category of $\dif^s$-modules on $\frB$ with free global
  presentations to that of coherent $\cO_\frO$-modules.
\item Thereunder, the functor $\Gamma(\frB; \underline{\:\:}) $
  corresponds to $\Gamma(\frO; \underline{\:\:}) $.
\end{enumerate}
\end{corollary}
\noindent
Under this equivalence, different components of $\frB$ correspond to
the characters of the $Z(\LD)$-gerbe over $\frO$ under
Corollary~\ref{deformationcor}.

\section{Computation of higher cohomologies }
\label{coh}
We now generalize the Beilinson--Drinfeld theorem~\ref{bdthm} by
describing the higher cohomology.

\subsection{The classical case.} 
We first extend Hitchin's calculation. Loosely phrased, our result
asserts that singularities in the fibers of the Hitchin map $\chi$ are
invisible to the cohomology of $\cO$. For simplicity, we assume here
that $G$ is simply-connected (so that $\LD$ is of adjoint type). The
result extends by matching components of $T^\vee\frB$ and characters
of $Z(\LD)$.

\begin{theorem}\label{classical}
  We have an isomorphism $H^\bullet(T^\vee\frB;\cO) \cong
  \Omega^\bullet[\mathbf{H}]$ with the algebraic differentials on the
  Hitchin base $\mathbf{H}$. The grading induced by the
  $\bC^\times$-action on $T^\vee$ comes from the natural grading on
  $\mathbf{H}$ on the even copy, shifted down by $1$ for the odd
  copy. This isomorphism depends on $k$ in the following way:
  rescalings of the restriction of $k$ to a simple factor of
  ${\mathfrak g}$ corresponds to rescalings of the appropriate summand
  in $\Omega^1[\mathbf{H}]$.
\end{theorem}

\begin{proof}
  The calculation parallels the arguments in \cite{fht}, especially
  \S8 and \S9.  We refer the reader to the latter for more background
  and the details of the methods (as they do involve
  infinite-dimensional vector bundles over a stack of infinite type).

  Choosing a point $\sigma\in \Sigma$ with a local coordinate $z$ and
  letting $\Sigma^o:=\Sigma\setminus\{\sigma\}$, we can present the
  stack $\frB$ as the double quotient $G[[z]]\backslash
  G((z))/G[\Sigma^o]$. We rewrite this as $G[[z]]\backslash
  \mathbf{X}$, with the \emph{thick flag variety} $\mathbf{X}:=
  G((z))/G[[\Sigma^o]]$.  However, we choose to present the tangent
  complex by a different, $2$-step resolution
\[
\frg[\Sigma^o] \xrightarrow{\partial=\mathrm{Ad}_\phi}
\frg((z))/\frg[[z]]
\]
with the differential equal to the adjoint twist of the obvious
inclusion at the point $\phi\in G((z))$.  The second term is a
trivial bundle over $\mathbf{X}$, whereas the first is associated to
the adjoint action of $G[\Sigma^o]$.

We seek the $G[[z]]$-equivariant hyper-cohomology of the differential
graded algebra
\begin{equation}\label{hypercoh}
  H^q\left(\frB;\mathrm{Sym}^r\, T\right) = \bigoplus_{s+t=r}
  \bH^q_{G[[z]]}\left(\mathbf{X};
    {\textstyle\bigwedge}^s\frg[\Sigma^o] \otimes \mathrm{Sym}^t
    (\frg((z))/\frg[[z]]) \right)
\end{equation}
with the generating bundle $\frg[\Sigma^o]$ placed in cohomological
degree $-1$ and differential induced from $\partial$. The
$\mathrm{Sym}$-factor carries the adjoint action of $G[[z]]$.

Filtering by $s$-degree gives a spectral sequence $E_1^{-s,q} \implies
\bH^{q-s}$ which we now compute.  Let us first recall why, as in \S9.4
of \cite{fht}, we have a \emph{key factorization} of the $E_1$ term,
\begin{equation}\label{e1split}
E_1^{-s,q} = \bigoplus_u H^{q-u}_{G[[z]]}\left(\mathrm{Sym}^{r-s}
  (\frg((z))/\frg[[z]])\right)\otimes
			H^{u}\left(\mathbf{X};
                          {\textstyle\bigwedge}^s\frg[\Sigma^o]\right).
\end{equation} 
The product above would \emph{a priori} be the $E_1$ term in the Leray
spectral sequence for the fibration $\frB\twoheadrightarrow
BG[[z]]$. However, we claim that the second factor
$H^\bullet(\mathbf{X}; \bigwedge^\bullet \frg[\Sigma^o])$ is a free
skew-commutative algebra, trivial as a $G((z))$-representation. If
so, then the Leray $E_1$ term is freely generated over the base by
classes which extend to the total space, and so \eqref{e1split} is the
true $E_1$ term for \eqref{hypercoh} and not just the first page of
its Leray sequence.

Our claim about $H^\bullet(\mathbf{X}; \bigwedge^\bullet
\frg[\Sigma^o])$ is a variant of \cite{fht}, Theorem~D, where we have
replaced the differentials $\Omega^1(\Sigma^o;\frg)$ in the exterior
algebra generators by the functions $\frg[\Sigma^o]$. Our functorial
construction of the generating cocycles in \cite{fht} \S9 allows in
fact for a twist by any line bundle $\mathscr{L}\to\Sigma^o$ in those
coefficients, with a corresponding twist appearing in the cohomology 
generators. More precisely, the general statement identifies
$H^{u}\left(\mathbf{X}; \bigwedge^s
  \Omega^1(\Sigma^o;\frg\otimes\mathscr{L}\right)$ with the free
bi-graded skew-commutative algebra generated by copies of
\begin{equation}\label{oddgen}
	\Gamma\left(\Sigma^o; \mathscr{L}^{\otimes
            m}\right)\quad\text{and}\quad 
		\Gamma\left(\Sigma^o; \mathscr{L}^{\otimes
                    (m+1)}\otimes K_\Sigma\right) 
\end{equation}
in degrees $(s,u) = (m,m)$ and $(m+1,m)$, as $m\in\exp(\frg)$. The
parity is given by $u-s$. Theorem~D of \cite{fht} had
$\mathscr{L}=\cO$.

In the present case, $\mathscr{L}=T\Sigma^o$, the spaces in each pair
\eqref{oddgen} are both isomorphic to \linebreak $\Gamma\left(\Sigma^o;
  (T\Sigma^o)^{\otimes m}\right)$, and with respect to
\eqref{e1split}, $s=r,u=q$, giving $(s,q,r)=(m,m,m), (m+1,m, m+1)$.

Similarly, Theorem~A of \cite{fht} (with the same twist in the
generators) describes the cohomology $H^q_{G[[z]]}\left(\mathrm{Sym}^t
  (\frg((z))/\frg[[z]])\right)$ as the skew-commutative algebra freely
generated by copies of
\begin{equation}\label{evengen}
  \left\{\bC((z))/\bC[[z]]\right\} \otimes (\partial/\partial
  z)^{\otimes m}
\end{equation}
placed in bi-degrees $(t,q)=(m,1)$ and $(m+1,0)$, as $m$ ranges over
$\exp(\frg)$. Parity is given by $q$. With $s=u=0$ in \eqref{e1split},
these correspond to $(s,q,r) = (0,1,m), (0,0,m+1)$.

The generators in \eqref{oddgen} and \eqref{evengen} come in pairs of
\emph{matching $r$-degrees} $m, m+1$, and with spectral sequence
bi-degrees $(-s,q)$ equal to $(-m,m)$ and $(0,1)$ for the former, and
$(-m-1,m)$ and $(0,0)$ for the latter. There is the possibility of
some obvious leading differentials in the spectral sequence for
\eqref{hypercoh} given by the restriction maps
\begin{equation}\label{resolve}
\Gamma\left(\Sigma^o; (T\Sigma^o)^{\otimes m}\right) \to 
		\left\{\bC((z))/\bC[[z]]\right\} \otimes
                (\partial/\partial z)^{\otimes m}
\end{equation}
from \eqref{oddgen} to \eqref{evengen}, on pages $m$ and $m+1$
respectively. That these are indeed the leading differentials can be
seen in the description of the generating cocycles via contraction
with the Atiyah class, as in \S8 and \S9 of \cite{fht}; the same
calculation applies here \emph{verbatim}.

In genus $\ge2$ and for semi-simple $\frg$ (when all $m>0$), the map
in \eqref{resolve} is injective with cokernel $H^1\left(\Sigma;
  (T\Sigma)^{\otimes m}\right)$. After resolving the leading
differentials, we thus get for $H^q\left(\frB; \mathrm{Sym}\,T\right)$
the free skew-commutative algebra on copies of
\[
H^1\left(\Sigma; (T\Sigma)^{\otimes m}\right),
\]
in cohomology degrees $q= 0,1$ and symmetric degrees $m+1$ and $m$,
respectively. Now, we \emph{know} from Hitchin's construction that
these generators survive to $E_\infty$ in the sequence: therefore, no
further differentials can occur and the theorem is proved. \end{proof}

\subsection{Quantization.}
We now deform the graded $\cO_{\frB}$-algebra $\mathrm{Sym}\, T\frB$
to the $\cO_{\frB}$-module $\dif^s$. The latter is filtered by the
order of the operator, and the associated graded sheaf is
$\mathrm{Sym}\, T\frB$. The canonical isomorphism
$H^\bullet(\frB;\dif^s)=\mathrm{Ext}_{\dif^s(\frB)}^\bullet\left(\dif^s;
\dif^s\right)$
makes the former into an algebra deforming $H^\bullet(\frB;
\mathrm{Sym}\, T)$.

The algebra of differential forms on opers, $\Omega^\bullet[\frO]$, is
also filtered, using the $\mathbf{H}$-affine structure on $\frO$ (see
Proposition~\ref{operfilt} and below). Recall that we have fixed a
non-degenerate invariant quadratic form $k$ on $\frg$. We then have an
isomorphism
$$\Omega^\bullet[\frO]\cong\mathrm{Ext}_\frL^\bullet\left(\cO_\frO;
  \cO_\frO\right)$$ which is obtained using the ($k$-dependent)
symplectic structure on $\frL$, which identifies the normal and
cotangent bundles to $\frO$.  The following supplies the higher
cohomologies in Theorem~\ref{bdthm}.

\begin{theorem}\label{quantum}
  There is a canonical isomorphism of filtered algebras
\[
H^\bullet(\frB;\dif^s) =
\mathrm{Ext}_\frL^\bullet\left(\cO_\frO;\cO_\frO\right).
\]
Using the ($k$-dependent) isomorphism with $\Omega^\bullet[\frO]$, the
associated graded map becomes the isomorphism of
Theorem~\ref{classical} (multiplied by $(-1)^{\on{deg}}$).
\end{theorem} 
\begin{remark}
  The use of the form $k^{-1}$ on ${}^L\frg$, concealed in the
  symplectic form on $\frL$, is the source of the $k$-dependence in
  Theorem~\ref{classical}.
\end{remark}

We prove the theorem in two steps. First, we partially refine our
calculation in Theorem~\ref{classical}:
\begin{proposition}\label{partialresult}
  The odd generators $\mathbf{H}^\vee$ of $H^1(\frB;\mathrm{Sym}\,T)$
  over $H^0$ have a preferred lift to $H^1(\frB;\dif^s)$. Furthermore,
  the associated graded algebra of $H^\bullet(\frB;\dif^s)$ (with
  respect to the filtration inherited from the order filtration on
  $\dif^s$) is naturally isomorphic to
  $H^\bullet(\frB;\mathrm{Sym}\,T)$.
\end{proposition}

\noindent
Then, in the next section, we ascertain the \emph{commutativity} of
$H^\bullet(\frB;\mathrm{Sym}\,T)$. A quantum version of the
construction in Theorem~\ref{classical} will present
$H^\bullet(\frB;\dif^s)$ as quotient of a \emph{commutative} algebra;
in fact, Theorem~\ref{ainfty} will refine Theorem~\ref{quantum} to an
isomorphism of $A_\infty$ algebras. This will identify the formal
neighborhood of $\dif^s$ within $\dif^s$-modules on $\frB$ with that
of $\cO_\frO$ inside $\cC oh(\frL)$, canonically to the first order.

\begin{proof}[Proof of Proposition \ref{partialresult}.]
  In Lemma~\ref{finite} below, we will show that the cohomologies of
  $\mathrm{Sym}\,T$ and $\dif^s$ on the (infinite type) stack $\frB$
  come from a \emph{finite type} substack.  Allowing this for now, the
  increasing order filtration on $\dif^s$, with $\mathrm{Gr}(\dif^s)
  = \mathrm{Sym}\,T$, leads to a spectral sequence with the first term
\[
E_1^{p,q} = H^{p+q}(\frB; \mathrm{Sym}^{-p}\,T) \implies
H^{p+q}(\frB;\dif^s).
\] 
We will show its collapse by verifying the survival of all $E_1$
generators. Theorem~\ref{bdthm} addresses $H^0(\frB;\dif^s) =
\bC[\frO]$. We will now construct a space of $H^1$ classes isomorphic
to $\mathbf{H}^\vee$, lifting the symbols of Hitchin's $H^1$
generators.

Those generators arise by contracting the linear (on $\mathbf{H}$)
Hamiltonian vector fields with the Chern class $c_1(\Theta_k)\in
H^1(\Omega^1)$ (Theorem~\ref{hitchinthm}). Now, if $c_1 = d\kappa$,
this would be the Poisson bracket with $\kappa$. There is no such
$\kappa\in H^1(\frB;\cO)$, but it \emph{does exist} in the
\emph{analytic} $H^1(\frB;\cO_{an}/\bC)$ (see \S\ref{diffstack}), as
seen from the long exact sequence
\[
\dots \to H^1(\frB;\cO_{an}) \to H^1(\frB;\cO_{an}/\bC) \to
H^2(\frB;\bC) \to H^2(\frB;\cO_{an}) \to \dots
\]
and the vanishing of $H^{>0}(\frB;\cO_{an})$.  This suffices for the
bracket interpretation, because $\bC$ is central: the Poisson bracket
lifts to a commutator operation
\[
[\,\underline{\:\:}\,,\,\underline{\:\:}\,]: H^*(\cO_{an}/\bC)\otimes
H^*(\dif^s) \to H^*(\dif^s),
\]
and similarly to any twisted version of $\dif$.

To avoid a GAGA comparison (which \emph{does} apply to the stack
$\frB$ but requires additional discussion), we describe this operation
in the algebraic category, where a similar operation exists for
classes in the Hodge filtered part $\bH^*\left(\Omega^1\xrightarrow{d}
  \Omega^2 \xrightarrow{d} \dots\right)$ of de Rham
cohomology. Morally, the operation is algebraic because the
\emph{commutator} of a differential operator with the
\emph{anti-derivative in} $\cO_{an}/\bC$ of an algebraic
hyper-cohomology class is algebraic. The precise definition, for a
class $\bH^1$, is the following.  Recall from \S\ref{diffstack} that
such a class $(\eta,\varphi)$ defines a twisted form
$\dif^{(\eta,\varphi)}$ of $\dif$. (We actually need to deform
$\dif^s$ but opt for less notational clutter.) We form the
$1$-parameter family $\dif^{t}$ over $\bC[[t]]$ defined by
$t(\eta,\varphi)$ and consider the long exact sequence
\[
	\dots \to H^p(\dif^t) \xrightarrow{t} H^p(\dif^t)
        \xrightarrow{\pi} H^p(\dif)
		\xrightarrow{\delta} H^{p+1}(\dif^t)
                \xrightarrow{t}\dots
\]
induced from $\dif^t\xrightarrow{t} \dif^t \xrightarrow{\pi}
\dif$. The composition $\pi\circ\delta: H^p(\dif)\to H^{p+1}(\dif)$ is
the desired operation. (The reader will recognize it as the first
differential in a spectral sequence $H^*(\dif)[[t]] \implies
H^*(\dif^t)$ for the deformation $\dif\leadsto \dif^t$.)  Applied to
our Chern class $c_1(\Theta_k)\in H^1(\frB;\Omega^1\to\Omega^2)$, this
lifts Hitchin's contraction on symbols, giving us the desired
generators in $H^1(\frB;\dif^s)$ and concludes the proof.
\end{proof}

\begin{remark}\label{collapse}
  On the Langlands dual side in Theorem~\ref{quantum}, moving away
  from Spin becomes the non-commutative deformation $\frL'$ of $\frL$
  generated by the symplectic form. In the deformed algebra, the
  algebra $\mathrm{Ext}^\bullet_{\frL'}\left(\cO_\frO;\cO_\frO\right)$
  collapses to $\bC$: the first differential of the spectral sequence
  of this deformation becomes de Rham's operator on
  $\Omega^\bullet[\frO]$, under the isomorphism in
  Theorem~\ref{quantum}, and the spectral sequence collapses to $\bC$
  on the second page.  So, for all twists $\dif'\neq\dif^s$ of $\dif$,
  $H^\bullet(\frB;\dif')\cong\bC$.
\end{remark}

We conclude with the technical result promised in the proof of the
proposition above.

\begin{lemma}\label{finite}
  The cohomologies $H^\bullet(\frB;\mathrm{Sym}\,T)$ and
  $H^\bullet(\frB; \dif')$ are computed correctly on some finite union
  of Atiyah-Bott strata of $\frB$, for any twisted form $\dif'$ of
  $\dif$.
\end{lemma}
\begin{proof}
  We will check the vanishing of local cohomologies on most strata. It
  will help to run the argument for $\mathrm{SL}(2)$ first, since the
  general case requires additional book-keeping.  An unstable
  Atiyah-Bott stratum $\frS$ of $\frB$ classifies bundles $E$ which
  are extensions
\[
0\to E' \to E \to E'' \to 0 
\]
with $\deg E' = -\deg E'' =d>0$. Choose $d\ge g$, and observe that
$\mathrm{Ext}^1_\Sigma(E'';E')=0$, so that $E\cong E'\oplus E''$. The
stratum is then a copy of the Jacobian $J$, modulo the semi-direct
product group $\mathrm{GL}(1)\ltimes \exp(\frh)$, with $\frh:=
\mathrm{Hom}(E'';E')$. The normal bundle is $\nu:=
\mathrm{Ext}^1(E';E'')$, with the obvious $\mathrm{GL}(1)$-action and
trivial action of $\frh$ (although the latter action is not trivial on
higher-order neighborhoods). The local cohomology sheaves
\begin{equation}\label{localcoh}
\mathscr{H}^*_{\frS}\left(\frB;\mathrm{Sym}\,T\right),\quad
\mathscr{H}^*_{\frS}(\frB;\dif')
\end{equation}
can be pushed down to the Jacobian $J$, where they have natural
associated graded sheaves which are resolved by the
$\mathrm{GL}(1)$-invariant part of the Chevalley complex for Lie
algebra cohomology of $\frh$
\begin{equation}
\textstyle{\bigwedge}^\bullet \frh^\vee \otimes \left(\det(\nu)\otimes
  \mathrm{Sym}\,\nu\right)\otimes \mathrm{Sym}\,\nu
\otimes \mathrm{Sym}\,TJ\otimes \textstyle{\bigwedge}^{-\bullet}
\left(\frh\oplus\mathfrak{gl}(1)\right);
\end{equation} 
the factor $\bigwedge\frh^\vee$ computes derived $\frh$-invariants,
$\det\nu\otimes\mathrm{Sym}\,\nu$ are the residues along $\frS$, and
the remaining factors come from the local resolution of the tangent
bundle of $\frB$.

The $\mathrm{GL}(1)$ weights on $\nu$ and $\frh^\vee$ are positive. In
genus $g\ge 2$, $\det(\nu)$ shifts the weights on $\bigwedge\frh$ to
the positive side as well (using Riemann-Roch and cohomology
vanishing), so there are no $\mathrm{GL}(1)$-invariants in
\eqref{localcoh} and the cohomologies with supports on $\frS$ vanish.
 
For a general group, we now repeat the argument with an additional
important observation: all Atiyah-Bott strata $\frS$, save for a
finite number of them, parametrize bundles $F\to\Sigma$ having a
destabilizing parabolic reduction $F_P$ to some $P\subset G$ (with
Levi component $L$, unipotent radical $U$) with the properties that
\[
H^0\left(\Sigma; \mathfrak{ad}_{F_P}(\fru^\vee)\right) = 0, \quad
H^1\left(\Sigma; \mathfrak{ad}_{F_P}(\fru)\right) = 0,
\]  
so that $F_P$ actually \emph{splits} to an $L$-bundle. Note that $F_P$
need \emph{not} be the maximal destabilizing reduction which defines
the stratum $\frS$: the latter will only work if the maximally
destabilizing coweight --- call it $\xi$ --- is sufficiently far
inside the Weyl chamber face which contains it. (Thus, regular
coweights must be some genus-dependent distance away from the walls,
whereas coweights on an edge must be some distance form $0$.)
Instead, we include in $\frp$ all root vectors $e_{\pm\alpha}$ for
simple roots $\alpha$ with $0\le \alpha(\xi) < 2g$. This partitions
the $\xi$'s into neighborhoods of the various faces, and we use the
parabolic associated to the smallest nearby face.

We can now repeat the $\mathrm{SL}(2)$ discussion, with 
\[
\frh:= H^0\left(\Sigma; \mathfrak{ad}_{F_P}(\fru)\right), 
\quad
\nu = H^1\left(\Sigma; \mathfrak{ad}_{F_P}(\fru^\vee)\right),
\] 
using a central $\mathrm{GL}(1)\subset L$ with positive weights on
$\fru$. The weight on $\det \nu$ dominates the one on $\det
\frh^\vee$, because
\[
H^0\left(\Sigma;\mathfrak{ad}_{F_P}(\fru)\right) \subset
H^0\left(\Sigma;\mathfrak{ad}_{F_P}(\fru)
	\otimes K_\Sigma\right) = H^1\left(\Sigma;
        \mathfrak{ad}_{F_P}(\fru^\vee)\right)^\vee = \nu^\vee
\]
whence we conclude the vanishing of cohomologies with supports.
\end{proof}

\begin{remark}
  Our results have analogues in genera $0$ and $1$, with a graded
  version of the Hitchin space with odd part $\bigoplus_m H^1(\Sigma;
  K_\Sigma^{\otimes (m+1)})$. Theorem~\ref{classical} applies, without
  any changes to the proof; however, the proof of Lemma~\ref{finite}
  does not go through.  Although we believe that Theorem~\ref{quantum}
  does hold, the technical deformation step seems to need a finer
  argument.
\end{remark}

\section{Commutativity to all orders and absence of obstructions}
  
We know that $H^\bullet(\frB;\dif^s)$ is generated by a copy of
$\mathbf{H}^\vee$ over $H^0\cong\bC[\frO]$, and that its associated
graded algebra is skew-commutative.  To finish the proof of
Theorem~\ref{quantum}, we need to verify the skew-commutativity of
$H^\bullet(\frB;\dif^s)$. We will prove a much stronger statement in
the same swoop. Namely, the graded algebra
\[
\mathrm{Ext}^\bullet_{\dif^s(\frB)}(\dif^s;\dif^s)
\]
carries the Yoneda $A_\infty$ multiplication. Complete knowledge of
this multiplication captures the formal deformations of $\dif^s$, and
also the localization of the derived category of Spin $\dif$-modules
at the object $\dif^s$.  We will show the strict commutativity of this
algebra.

\begin{theorem}\label{ainfty}
  The Yoneda algebra $\mathrm{Ext}_{\dif^s(\frB)}(\dif^s;\dif^s)$ is
  abstractly $A_\infty$-isomorphic to (the strictly skew-commutative one)
  $\mathrm{Ext}^\bullet_\frL\left(\cO_\frO;\cO_\frO\right)$.
\end{theorem}
\noindent
In particular, formal deformations of $\dif^s$ as a $\dif^s$-module
are unobstructed to all orders and correspond to deformations of
$\frO\subset\frL$. Theorem~\ref{ainfty} extends
Corollary~\ref{deformationcor} as follows.

\begin{corollary}
\begin{enumerate}
\item We have an equivalence between the thick closure of $\dif^s$ in
  $\dif^s_{\frB}\mbox{-mod}$ and the full subcategory of sheaves in
  $\cC oh(\frL)$ with set-theoretic support on $\frO$.
\item The derived categories of the above are equivalent to the thick
  closures of $\dif^s$ and $\cO_\frO$ in the derived categories of
  $\dif^s_{\frB}\mbox{-mod}$ and $\cC oh(\frL)$, respectively.
\end{enumerate}
\end{corollary}

\begin{remark}
% Theorem~\ref{ainfty} does not provide all that we would like.
  The earlier identification of $H^0, H^1(\frB;\dif^s)$ with
  $\mathrm{Ext}^0, \mathrm{Ext}^1 (\cO_\frO)$ gives only the first map
  in an $A_\infty$ isomorphism. Loosely speaking, we have identified
  the first neighborhood of $\dif^s$ in $\dif^s\mbox{-mod}$ with that
  of $\frO$ in $\frL$, and are now proving that the identification can
  be continued to all orders. We intend to address that and pin the
  isomorphism down canonically in a follow-up paper.
\end{remark}

The proof of Theorem~\ref{ainfty} relies on its \emph{local} version,
in which $\Sigma$ is replaced by a formal disk.
%(or, more precisely, the ``formal disk relative its boundary'').
The relevant disk $D$ is the formal neighborhood of
$\sigma\in\Sigma$. Namely, Theorem~4.4 from our previous paper
\cite{ft} gives an isomorphism of skew-commutative algebras
\begin{equation}\label{local}
H^\bullet(\frg[[z]],\frg;\mathbb{V}_{\on{crit}}) \cong \Omega^\bullet[\frO(D)]
\end{equation}
where on the left we have the Lie algebra cohomology with coefficients
in the \emph{vacuum module at critical level} $\mathbb{V}_{\on{crit}}:
= U_{\on{crit}}\,\widehat\frg\otimes_{\frg[[z]]}\bC$ (see \cite{ft}
\S2 for the definitions), while on the right $\frO(D)$ is the variety
of opers on the formal disk. The algebra structure on the left side is
explained by the isomorphisms
\[
H^\bullet(\frg[[z]],\frg;\mathbb{V}_{\on{crit}}) \cong
H^\bullet\left(\frg((z)),\frg;\mathrm{End}(\mathbb{V}_{\on{crit}})\right) \cong
\mathrm{Ext}^\bullet_{HC(\wh\frg,G[[z]])}
\left(\mathbb{V}_{\on{crit}};\mathbb{V}_{\on{crit}}\right),
\]
which are proved in \cite{ft}. The $\mathrm{Ext}$ groups are
computed in a suitably defined category of the Harish-Chandra modules
for the pair $(\widehat\frg,G[[z]])$ of critical level, and the
endomorphism algebra as well as the Lie algebra cochains must be
suitably completed. We refer to \cite{ft} for the requisite technical
details.

For maximum consequence, we need the $A_\infty$ enhancement of
\eqref{local}.
\begin{proposition}\label{localainfty}
  The isomorphism \eqref{local} may be enriched to one of $A_\infty$
  algebras, with the Yoneda structure on the left and the
  skew-commutative structure on the right.
\end{proposition}
\begin{proof}
The degree zero part of the isomorphism \eqref{local},
\[
H^0(\frg[[z]]; \mathbb{V}_{\on{crit}}) = \mathrm{End}_{\wh\frg} 
	\left(\mathbb{V}_{\on{crit}};
          \mathbb{V}_{\on{crit}}\right)\cong \bC[\frO(D)],
\] 
was established in \cite{ff}, where it is shown to represent the
``negative Fourier half'' of the center of the (critically twisted)
universal enveloping algebra $U_{\on{crit}}\,\wh\frg$, which implies
that $\mathrm{Ext}^0$ is \emph{strictly central} in
$\mathrm{Ext}^\bullet$: that is, the latter is an $A_\infty$ algebra
over $\mathrm{Ext}^0$.

Now, $\mathrm{Ext}^\bullet$ is freely generated in degree $1$ over
$\mathrm{Ext}^0$.  The $A_\infty$ structure on $\mathrm{Ext}^\bullet$
is a deformation of the standard skew-commutative one, deformation
being implemented by degree-scaling.  It is easier to describe this in
a Koszul dual picture, so let us move to the Koszul dual of
$\mathrm{Ext}^\bullet$ over $\mathrm{Ext}^0$.  This is a deformation
of a commutative algebra based on the cotangent bundle of $\frO(D)$,
and it suffices to show that this deformation is trivial. The only
possible deformation is to a non-commutative structure, implemented to
the leading order by a Poisson bi-vector field along the fibers of
$T^\vee(\frO(D))$.

Here is the key observation: any such deformation will obstruct, at
some order, displacements of the module supported at the zero-section
in the cotangent bundle. Indeed, this is just saying that any
deformation of the algebra $A$ of functions on (the neighborhood of
$0$ in) a vector space will obstruct, at some order, the deformation
of the skyscraper sheaf at $0$. For, otherwise, we can identify the
completion of $A$ with the algebra of functions on that (unobstructed)
deformation space, and hence infer the commutativity of $A$. But we
will now show that the zero-section can be displaced, to all orders,
in the direction of any exact one-form in $\Omega^1[\frO(D)]$ (note
that we can choose all generators in $\Omega^1[\frO(D)]$ to be exact).
This will prove the strict commutativity of the
$\mathrm{Ext}$-algebra.

Clearly, it suffices to carry out the deformation in the original
algebra $\mathrm{Ext}$.  (Koszul duality was merely a convenient
depiction of the argument above.)  The zero-section structure sheaf is
$\mathrm{End}\left(\mathbb{V}_{\on{crit}}\right)$. We will now see
that $\mathbb{V}_{\on{crit}}$ can be deformed to all orders with any
initial direction corresponding to an exact form in
$\Omega^1[\frO(D)]$; then, $\mathrm{Hom}_{\wh\frg}\left(
  \mathbb{V}_{\on{crit}}; \mathbb{V}_{\on{crit}}^{\mathrm{def}}
\right)$ will provide the deformation of the zero-section, proving the
proposition.

Choose $a\in\bC[\frO(D)]$. Consider the family of universal enveloping
algebras $U_{c+\vep k}\,\wh\frg$ parametrized by the varying level
$c+\vep k$ away from critical, and choose a family $z(\vep)\in
U_{c+\vep k}\,\wh\frg$, with $z(0)$ a central element corresponding to
$a$. We consider the formal $1$-parameter family of (critical level)
$\wh\frg$-module structures $\rho_t$ on $\mathbb{V}_{\on{crit}}$
deforming the standard one $\rho = \rho_0$, defined by
\[
\rho_t(\xi) :=  \lim_{\vep\to 0}
\left[\exp\left(\frac{tz(\vep)}{\vep}\right) \rho(\xi)
		\exp\left(-\frac{tz(\vep)}{\vep}\right)\right]. 
\]
Because $z(0)$ is central, the limit exists. Clearly, $t=0$ recovers
$\mathbb{V}_{\on{crit}}$. The first order term (in $t$) of the
deformation of $\rho$ is defined by the formula
\[
\lim_{\vep\to 0} \frac{1}{\vep}\left[z(\vep),\xi\right].
\]
As shown in \cite{ft}, (4.3), this class in
$\mathrm{Ext}^1_{HC}\left(\mathbb{V}_{\on{crit}},
  \mathbb{V}_{\on{crit}}\right)$ corresponds to $da$ under the
isomorphism \eqref{local}.
\end{proof}

\begin{remark}
  The first order deformation above is the bracket with $z$ with
  respect to the Poisson structure on the center of
  $U_{\on{crit}}\,\wh\frg$ \cite{ff}.
\end{remark}

We conclude with two (closely related) proofs of Theorem~\ref{ainfty},
both of which rely on its local analogue, Proposition~\ref{localainfty}.

%Note that the only improvement claimed over
%Theorem~\ref{bdthm} and Proposition~\ref{partialresult} is (full)
%commutativity, so this is what we will prove.

\begin{proof}[First proof of Theorem~\ref{ainfty}]
  We exhibit the Yoneda $\mathrm{Ext}$-algebra as a quotient of the
  strictly commutative one
  $H^\bullet(\frg[[z]],\frg;\mathbb{V}_{\on{crit}})$. Let us choose in
  it a vector subspace of generators of
  $H^\bullet(\frg[[z]],\frg;\mathbb{V}_{\on{crit}})$. Then this
  implements an $A_\infty$ isomorphism of
  $H^\bullet\left(\frB;\dif^s\right)$ with a skew-commutative
  sub-algebra of $H^\bullet (\frg[[z]],\frg;\mathbb{V}_{\on{crit}})$.

  Returning to the presentation $\frB= G[[z]]\backslash \mathbf{X}$,
  the following ``quantum'' version of the construction of
  $\mathrm{Sym}\,T$, from the proof of Theorem~\ref{classical},
  presents instead the cohomology of $\dif^s$:
\begin{equation}\label{quantumhypercoh}
H^q\left(\frB;\dif^s\right) = \bH^q_{G[[z]]}\left(\mathbf{X};
  \left({\textstyle\bigwedge}^\bullet
\frg[\Sigma^o] \otimes \mathbb{V}_{\on{crit}}; \partial\right) \right);
\end{equation}
$\partial$ is the Chevalley differential for the fiber-wise Lie
algebra action of $\frg[\Sigma^o]$ on $\mathbb{V}_{\on{crit}}$,
twisted at the point $\phi\cdot G[\Sigma^o]\in \mathbf{X}$ by the
adjoint action of the loop group element $\phi$. If the genus of $\Sigma$
is $2$ or greater, this complex is a resolution of its bottom (that
is, zeroth) homology, so is in fact equivalent to a $\dif^s$-module.
Taking symbols recovers the ``classical'' presentation in
\eqref{hypercoh}.

There is an obvious action on \eqref{quantumhypercoh} of the center of
the completed enveloping algebra $U_{\on{crit}}\,\wh\frg$
\cite{ff}. This center surjects onto $H^0(\frB;\dif^s)$, and the
resulting description of $\Gamma(\frB;\dif^s)$ is just the
Beilinson--Drinfeld Theorem~\ref{bdthm}. More precisely, the action of
the center factors through the action on \eqref{quantumhypercoh} of
the algebra $\mathrm{End}_{\wh\frg}
\left(\mathbb{V}_{\on{crit}}\right)$, the degree zero part of
\eqref{local}. When identified with functions on $\frO(D)$, central
elements act via their restriction to the subvariety $\frO$ of opers
on $\Sigma$.

We want to extend this action to the entire Yoneda
$\mathrm{Ext}$-algebra of $\mathbb{V}_{\on{crit}}$.  More precisely,
to any higher Yoneda self-extension of $\mathbb{V}_{\on{crit}}$ as a
projective Harish-Chandra module for the pair $(\wh\frg,G[[z]])$,
\begin{equation}\label{yonext}
\mathbb{V}_{\on{crit}} \to E_1 \to \dots \to E_k \to
\mathbb{V}_{\on{crit}} 
\end{equation}
we can functorially associate a self-extension of $\dif^s$ by using
the presentation in \eqref{quantumhypercoh}. When $\Sigma$ has genus
$2$ or greater, this gives a surjection of
$H^\bullet(\frg[[z]],\frg;\mathbb{V}_{\on{crit}})$ onto
$H^\bullet\left(\frB;\dif^s\right)$, because we already get a
surjection at the level of symbols.

A potential difficulty is that the complex in \eqref{quantumhypercoh},
applied to an arbitrary element of our $HC$ category, need not be
concentrated in degree zero. In other words, the functor described by
the degree-zero part of the complex of sheaves in
\eqref{quantumhypercoh}, as a functor from projective
$\wh\frg$-modules to $\dif^s$-modules, is not exact on the full
Harish-Chandra category. It is, however, exact on the full, exact
subcategory $EV$ consisting of finite, successive extensions of
$\mathbb{V}_{\on{crit}}$. (The objects of $EV$ are modules which admit
a finite filtration with associated graded a direct sum of copies of
$\mathbb{V}_{\on{crit}}$.)  Induction on the length shows that on
$EV$, the localization functor \eqref{quantumhypercoh} is concentrated
in degree zero (and exact), so it gives an $A_\infty$ morphism of
$\mathrm{Ext}$-algebras.  Now, because
$\mathrm{Ext}(\mathbb{V}_{\on{crit}};\mathbb{V}_{\on{crit}})$ in the
Harish-Chandra category is generated by $\mathrm{Ext}^1$, and the
commutativity relations are enforced by formal deformations which (to
all orders) belong to $EV$, we may assume that the $E_i$ in
\eqref{yonext} belong to $EV$. In other words, the $\mathrm{Ext}$
group computed in $EV$ agrees with the one in the Harish-Chandra
category. (For a discussion of Yoneda extension classes in exact
categories, see for example \cite{nr}.) With this observation, we do
get the desired $A_\infty$ surjection.
\end{proof}

\begin{proof}[Second proof of Theorem~\ref{ainfty}]
  This argument follows the idea of the proof of
  Proposition~\ref{localainfty}; so we will not repeat all details,
  but only address the parts which need modification. By means of the
  construction \eqref{quantumhypercoh}, the formal deformations of
  $\mathbb{V}_{\on{crit}}$ described in the proof of Thm.~\ref{ainfty}
  lead to formal deformations of $\dif^s$ as $\dif^s$-modules. This
  shows the existence of formal deformations of $\dif^s$ to all
  orders, tangent to any direction in $\mathrm{Ext^1} =
  \Omega^1[\frO]$ which corresponds to an \emph{exact} differential.

  As before, we Koszul-dualize $\mathrm{Ext}_{\dif^s}\left(\dif^s;
    \dif^s\right)$ to obtain a deformation of the formal neighborhood
  of $\frO$ within its cotangent bundle. Again, to leading order, this
  deformation is given by a Poisson bi-vector field. The key
  observation, which concludes the proof, is that \emph{the only
    Poisson structure under which first-order Lagrangian displacements
    of the zero-section are unobstructed to all orders is a constant
    multiple of the standard symplectic form}. This is shown in the
  Lemma below. However, deformation in that particular direction would
  collapse the $\mathrm{Ext}$-algebra to $\bC$: see
  Remark~\ref{collapse}. This is disallowed by our calculation of
  $\mathrm{Ext}^\bullet$: so the Poisson structure vanishes, our
  Koszul dual algebra is strictly commutative, and then so is the
  $A_\infty$ structure on the original $\mathrm{Ext}$-algebra.
\end{proof}

\begin{lemma}
  Let $X\subset Y$ be manifolds (affine algebraic or Stein) of
  dimension $2$ or more, such that the normal bundle of $X$ is
  identified with $T^\vee X$. Let $\alpha$ be a non-trivial Poisson
  structure near $X\subset Y$, for which $X$ is involutive. Assume
  that first-order Lagrangian displacements of $X$ are unobstructed to
  all orders: that is, for every closed differential $\varphi$ on $X$,
  there exists a formal $1$-parameter family of $\alpha$-involutive
  submanifolds of $Y$, deforming $X$, and equal to first order to the
  graph of $\varphi$. Then, in a formal neighborhood of $X$, $Y\cong
  T^\vee X$ in such a way that $\alpha$ is the standard Poisson structure.
\end{lemma}

\begin{proof} 
  Degenerate $Y$ to $T^\vee X$ and retain the leading part
  $\mathrm{gr}(\alpha)$ of $\alpha$. The graphs of all differentials
  are now involutive. Now, it is easy to check that, on the standard
  symplectic vector space, a constant Poisson structure for which all
  standard Lagrangian subspaces are involutive is a multiple of the
  symplectic Poisson structure. So $\mathrm{gr}(\alpha)$ is a
  (point-dependent) multiple of the symplectic Poisson structure on
  $T^\vee X$. But only the \emph{constant} multiples of the standard
  Poisson structure on $T^\vee X$ are Poisson, if $\dim X>1$. So
  $\mathrm {gr}(\alpha)$ is a constant multiple of the standard
  Poisson structure. But then, $Y$ is symplectic near $X$, with $X$
  Lagrangian and the Lemma becomes the formal version of Weinstein's
  tubular neighborhood theorem.
\end{proof}

%\noindent
%\small
%{E.~Frenkel, \texttt{frenkel@math.berkeley.edu} \\
%C.~Teleman, \texttt{teleman@math.berkeley.edu}\\
%Department of Mathematics, 970 Evans Hall \#3840, Berkeley, CA 94720, USA
%}

{\footnotesize \noindent 
DEPARTMENT OF MATHEMATICS, UNIVERSITY OF CALIFORNIA,
  BERKELEY, CA 94720, USA}

\end{document}